\documentclass{llncs}

\usepackage{amsmath,amssymb}

\usepackage{comment}

\newcommand{\R}{\mathbb{R}}
\newcommand{\Z}{\mathbb{Z}}
\newcommand{\N}{\mathbb{N}}
\newcommand{\CA}{\mathcal{CA}}

\newcommand{\B}{\mathcal{B}}
\newcommand{\ID}{\mbox{id}}
\newcommand{\INF}{{}^\infty}

\begin{document}

\title{Geometry and Dynamics of the Besicovitch and Weyl Spaces\thanks{Research supported by the Academy of Finland Grant 131558}}
\titlerunning{Geometry and Dynamics of the Besicovitch and Weyl Spaces}

\author{Ville Salo\inst{1} \and Ilkka T\"{o}rm\"{a}\inst{2}}
\authorrunning{V. Salo and I. T\"{o}rm\"{a}}

\institute{University of Turku, Finland,\\
Turku Centre for Computer Science, Finland,\\
\email{vosalo@utu.fi}
\and
University of Turku, Finland,\\
\email{iatorm@utu.fi}}

\maketitle

\begin{abstract}
We study the geometric properties of Cantor subshifts in the Besicovitch space, proving that sofic shifts occupy exactly the homotopy classes of simplicial complexes. In addition, we study canonical projections into subshifts, characterize the cellular automata that are contracting or isometric in the Besicovitch or Weyl spaces, study continuous functions that locally look like cellular automata, and present a new proof for the nonexistence of transitive cellular automata in the Besicovitch space.
\end{abstract}

\setcounter{page}{1}

\section{Introduction}

In the field of symbolic dynamics and cellular automata, the Besicovitch and Weyl topologies (called \emph{global topologies} in this article) have become objects of profound interest. The Besovocitch space was introduced in \cite{CaFoMaMa97} to study the chaoticity of cellular automata. While the dynamical properties of CA on the global spaces have been researched to some extent, not much is known about the geometry of Cantor subshifts and cellular automata in the global spaces. This is an interesting direction of research, since in the zero-dimensional Cantor topology, we cannot really talk about the `shape' of objects. 

In Section~\ref{sec:Topology} we prove some basic topological results about subshifts with the Besicovitch topology. In particular, we show the contractibility of all mixing sofic shifts, and give a natural sufficient condition for two mappings to a sofic shift to be homotopic. The proofs for these results are extensions of the techniques used in \cite{BlFoKu97} and \cite{DoIw88} to prove the path-connectedness of the full shift.

In Section~\ref{sec:SimpComp}, we refine the above results, proving that sofic shifts exhibit exactly the same homotopy equivalence classes as simplicial complexes. We then show an example of a natural zero entropy subshift which is not homotopy equivalent to any simplicial complex.

One of the constructions of Section~\ref{sec:SimpComp} relies on the idea of `projecting' a sofic shift onto one of its sofic subshifts. In Section~\ref{sec:Projections}, we show that no nontrivial subshift $X$ is `convex' in the global spaces, in the sense that every element of the full shift would have a unique nearest approximating point in $X$. This contrasts the situation of simplicial complexes, where simplexes have unique nearest points in each of their faces.

In Section~\ref{sec:CA}, we consider $\N$-actions on subshifts with the global topologies given by cellular automata. We present a class of subshifts on which all cellular automaton having contracting, isometric or expanding dynamics have neighborhood size $1$. This is a result in the same spirit as the fact that cellular automata cannot be expansive on the full shift with the Besicovitch topology \cite{BlFoKu97}. We also present an unrelated result that on the full shift, cellular automata are exactly those continuous functions whose image at each point is computed by some (possibly different) cellular automaton. The main interest of our proof is the use of a topology on the space of all cellular automata.

In Section~\ref{sec:DynamicalProperties}, we give a new proof for the nonexistence of transitive cellular automata on the Besicovitch space.

In our proof, the use of Kolmogorov complexity, used to prove the result in \cite{BlCeFo03}, is replaced by a measure theoretical argument, and to the best of our knowledge, this is a first application of measure theory in the study of the global spaces. Finally, we construct a shift-invariant nonatomic measure on the full shift, the existence of which was asked in \cite{CaFoMaMa97}. The support of this measure is a relatively unnatural zero-entropy subshift, and it is in fact easily seen that no measure can be fully supported on the full shift.

\section{Definitions}

Let $\Sigma$ be a finite \emph{state set}, and denote by $H(u, v)$ the Hamming distance between two words $u, v \in \Sigma^*$ of equal length. We define three different pseudometrics on the \emph{full shift} $\Sigma^\Z$. The \emph{Cantor topology} is given by the metric $d_C(x, y) = 2^{-\delta}$ where $\delta = \min\{|i| \;|\; x_i \neq y_i\}$, the \emph{Besicovitch topology} by the pseudometric
\[ d_B(x, y) = \limsup_{n \in \N} {\frac{H(x_{[-n, n]}, y_{[-n, n]})}{2n + 1}}, \]
and the \emph{Weyl topology} by the pseudometric
\[ d_W(x, y) = \limsup_{n \in \N} \max_{m \in \Z} {\frac{H(x_{[m - n, m + n]}, y_{[m - n, m + n]})}{2n + 1}}. \]
The \emph{shift map} $\sigma : \Sigma^\Z \to \Sigma^\Z$ defined by $\sigma(x)_i = x_{i+1}$ is easily seen to be continuous in all the three topologies, and isometric in the Besicovitch and Weyl spaces. For two configurations $x,y \in \Sigma^\Z$, we denote $x \sim_B y$ if $d_B(x,y)=0$, and $x \sim_W y$ if $d_W(x,y)=0$. The letters $B$ and $W$ may be suppressed if they are clear from the context. In general, for each topological concept (when the meaning is obvious) we use the terms \emph{C-concept}, \emph{B-concept} and \emph{W-concept} when the Cantor, Besicovitch or Weyl topology, respectively, is used on $\Sigma^\Z$. The term \emph{G-concept} is used when the discussion applies to both the B-concept and the W-concept. If the prefix is omitted and no other indication is given, we always mean the corresponding C-concept. We collectively refer to the Besicovitch and Weyl topologies as the \emph{global topologies}. If $X \subset \Sigma^\Z$, we define the \emph{G-projection of $X$} as $\tilde X = \{ y \in \Sigma^\Z |\; \exists x \in X: x \sim_G y \}$.

In a pseudometric space $X$, for all points $x \in X$ and $\epsilon > 0$, we denote $B_\epsilon(x) = \{y \in X \;|\; d(x,y) < \epsilon \}$. For a (finite or infinite) word $w$, we denote by $w^R$ the reversal or $w$. For two words $v$ and $w$, we denote $v \sqsubset w$ and say that $v$ \emph{occurs in $w$}, if $v = w_{[i,i+|v|-1]}$ for some position $i$. For a set $X$ of words, $v \sqsubset X$ means that $v \sqsubset w$ for some $w \in X$. The notation $|w|_v$ denotes the number of occurrences of $v$ in $w$.

We assume a basic aqcuiantance with the notions of symbolic and topological dynamics. A clear presentation is given in \cite{LiMa95} and \cite{Ku03}, respectively.

Let $X$ and $Y$ be topological spaces. Two continuous functions $f, g : X \to Y$ are said to be \emph{homotopic}, denoted $f \sim g$, if there exists a continuous function $h : [0, 1] \times X \to Y$ such that $h(0, x) = f(x)$ and $h(1, x) = g(x)$ for all $x \in X$, and $h$ is then a \emph{homotopy} between $f$ and $g$. For $Z \subset X$, we say $f, g : X \to Y$ are \emph{homotopic relative to $Z$} if there exists a homotopy $h$ between $f$ and $g$ such that $h(r, z) = h(0, z)$ for all $r \in [0, 1], z \in Z$. We say two spaces $X$, $Y$ are \emph{homotopy equivalent} if there exist two continuous functions $f : X \to Y$ and $g : Y \to X$ such that $g \circ f \sim \ID_X$ and $f \circ g \sim \ID_Y$. We say that a space is \emph{contractible} if it is homotopy equivalent to the trivial space $\{0\}$.

An \emph{$n$-simplex} is an $n$-dimensional polytope which is the convex hull of its $n+1$ vertices. A \emph{simplicial complex} is a collection $K$ of simplices such that any face of an element of $K$ is in $K$, and any intersection of two simplices of $K$ is their common face. An \emph{abstract simplicial complex} with vertices in a finite set $V$ is a collection $K \subset 2^V$ closed under subsets. The corresponding complex can be realized in $\R^{|V|}$ by mapping $V$ via $f$ to a set of linearly independent points, and for all $T \in K$, including the convex hull of $f(T)$. The \emph{$k$-skeleton} of a complex is the subcomplex formed by its simplexes of dimension at most $k$. A simplicial complex (abstract or not) is frequently identified with the corresponding (realized in the case of an abstract complex) subset of some Euclidean space.

\section{Basic Topological Tools and Results}
\label{sec:Topology}

One of the first results on the full shift with a global topology was its path-connectivity, proved in \cite{FoKu98} by explicitly constructing a path between $\INF 0 \INF$ and $\INF 1 \INF$ in $\{0,1\}^\Z$. We present this construction, and then show another simpler way to construct paths in the Besicovitch space.

\begin{definition}
\label{def:WPath}
We define the function $U : [0, 1] \to \{0, 1\}^\N$ in the following way: For $x \in [0, 1)$, let $D(x) \in \{0, 1\}^\N$ be a binary extension of $x$ that does not end with infinitely many $1$'s. We let $y^0 = (0*) \INF, y^1 = (*1) \INF$ be points of $X = \{0, 1, *\}^\N$.

We define the interspersing operation $I : X \times X \to X$ by
\[
I(x, y)_i = \left\{\begin{array}{ll}
	y_j,				& \mbox{if } x_i = * \mbox{ and } |x_{[0, i)}|_* = j \\
	x_i,				& \mbox{otherwise}
\end{array}\right.,
\]
the sequence $(x^i)_{i \in \N}$ by $x^0 = * \INF, x^{i + 1} = I(x^i, y^{D(x)_i})$, and finally $U(x)$ as the point $y \in \{0, 1\}^\N$ which is the C-limit of the sequence $(x^i)_{i \in \N}$. It is easy to see that this limit exists, and no $*$ appears in it. We define $U(1) = 1 \INF$.

We define $T : [0, 1] \to \{0, 1\}^\Z$ by $T(x) = U(x)^R.U(x)$.
\end{definition}

It is easy to see that $T(0) = \INF 0 \INF$, and that both $T$ and $U$ are G-continuous.

\begin{definition}
\label{def:BPath}
We define the function $U' : [0, 1] \to \{0, 1\}^\N$ in the following way: Partition $\N$ into the intervals $[2^{n-1}-1, 2^n-1)$ for $n > 0$. The resulting partition is called $\mathcal{N}$. Then, define $U'(x)$ by filling each $[a, b) \in \mathcal{N}$ with $0^{\lfloor x(b-a) \rfloor}1^{\lceil (1-x)(b-a) \rceil}$. We define $T' : [0, 1] \to \{0, 1\}^\N$ by $T'(x) = U'(x)^R.U'(x)$.
\end{definition}

B-continuity is again easy to verify, but the functions are clearly not W-continuous. In this and the following section, we modify paths obtained in Definition~\ref{def:WPath} and Definition~\ref{def:BPath} to find paths between points of a sofic shift.

The following result, found as Theorem~2.2.3 in \cite{Ma96}, will be useful later:

\begin{lemma}[\cite{Ma96}]
\label{lemma:LiesImplyHomotopy}
Let $Y \subset \R^n$, and let $f, g : X \to Y$ be functions. If for each $x \in X, r \in [0, 1]$ we have $rf(x) + (1 - r)g(x) \in Y$, then $f$ and $g$ are homotopic relative to $\{x \;|\; f(x) = g(x)\}$.
\end{lemma}

We will also prove a Besicovitch space version of Lemma~\ref{lemma:LiesImplyHomotopy} for functions mapping to sofic shifts. From this lemma, it is easy to conclude that a mixing sofic shift is B-contractible, that is, B-homotopy equivalent to a trivial simplex. We will refine this result in the next section.

Let $X$ be a sofic shift and $X_1, X_2$ two of its transitive components. The set of points that are left asymptotic to a point of $X_1$ and right asymptotic to a point of $X_2$ is called a \emph{bitransitive component of $X$}. We note that $X$ is covered by its bitransitive components, but they are not necessarily subshifts. For simplicity's sake, we will only work with one-way subshifts (subsets of $\Sigma^\N$) and assume that all transitive components are mixing until the beginning of Section~\ref{sec:Projections}. The generalization to arbitrary two-way subshifts is not difficult, when one replaces transitive components with bitransitive ones in the results, and accounts for the nonmixingness in a suitable way.

\begin{definition}
For a mixing sofic shift $X$ with mixing distance $m$, two points $x,y \in X$ and $r \in [0,1]$, define the point $A_X(r,x,y) \in X$ as follows. For all $i \in \N$ such that $i \in [a+m, b-m)$ for some $[a, b) \in \mathcal{N}$, we define
\[
A_X(r, x, y)_i = \left\{\begin{array}{ll}
	x_i,				& \mbox{if } U'(r)_{[i, i+m)} = 0^m \\
	y_i,				& \mbox{if } U'(r)_{(i-m, i]} = 1^m
\end{array}\right..
\]
The part left undefined has zero density, and can be filled to obtain a point of $X$. It is clear that the resulting \emph{average function} $A_X : [0,1] \times X^2 \to X$ is B-continuous, and we extend it to $\tilde X$ in a continuous way.
\end{definition}

Using the average function $A_X$, the following version of Lemma~\ref{lemma:LiesImplyHomotopy} for the Besicovitch space is easy to prove:

\begin{lemma}
\label{lemma:BHomotopy}
Let $Y$ be a sofic shift and let $f : X \to Y$ and $g : X \to Y$ be B-continuous functions. If for all $x \in X$, there exists a transitive component $Z_x$ of $Y$ such that $f(x) \in \tilde Z_x$ and $g(x) \in \tilde Z_x$, then the functions $f$ and $g$ are B-homotopic.
\end{lemma}

\begin{proof}
Define the homotopy $h : X \times [0,1] \to Y$ by $h(x,r) = A_{Z_x}(r, f(x), g(x))$. Note that by the way we defined $A_X$, the choice of the component $Z_x$ does not affect the continuity of $h$. \qed
\end{proof}

\begin{corollary}
\label{corollary:BSoficCont}
A mixing sofic shift is B-contractible.
\end{corollary}

\begin{theorem}
\label{theorem:SoficClosed}
Let $X$ be a sofic shift. Then $\tilde{X}$ is G-closed.
\end{theorem}

\begin{proof}
We prove the theorem in the Weyl case, the Besicovitch case being even easier. It is enough to show that $\tilde{Y}$ is W-closed for every mixing component $Y$ of $X$, so assume on the contrary that $x \notin \tilde{Y}$ and $d_W(x, Y) = 0$ for some $x \in \Sigma^\Z$. Let $k$ be the mixing distance of $Y$, and build a new point $y \in Y$ as follows. Let $w_0$ be the maximal prefix of $x$ such that $w_0 \sqsubset Y$. Let $y_{[0,|w_0|-1]} = w_0$, and let $w_1$ be the maximal prefix of $x_{[|w_0|+k,\infty)}$ such that $w_1 \sqsubset Y$. Let $y_{[|w_0|+k,|w_0|+k+|w_1|-1]} = w_1$, and continue the process infinitely. The part $D \subset \N$ left undefined can be filled arbitrarily using mixingness so that $y \in Y$.

We claim that $d_W(x,y) = 0$. If this were not the case, we would have an $\epsilon > 0$ and a sequence of pairs $(m_i,n_i) \in \N^2$ such that $n_i \longrightarrow \infty$ and
\[ |D \cap [m_i,m_i+n_i-1]| \geq \epsilon n_i + 2k \]
for all $i$. Since $d_W(x, Y) = 0$, there exists $z \in Y$ with $d_W(x,z) < \frac{\epsilon}{2(k+1)}$. Choose $N \in \N$ such that $H(x_{[m,m+n-1]},z_{[m,m+n-1]}) < \frac{\epsilon n}{k+1}$ for all $m \in \N$ and $n \geq N$. Then choose $i$ such that $n_i \geq N$.

Partition the segment $S = [m_i,m_i+n_i-1]$ to the at most $\frac{\epsilon n_i}{k+1} + 2$ subsegments whose first coordinates are $m_i$ and $\{ j \in S \;|\; x_j \neq z_j \}$. By the definition of $y$, we have $|D \cap s| \leq k$ for all $s = [a,b] \in S$: if $j = \min D \cap s$, then $x_{[j+k,b]} \sqsubset Y$, and thus $D \cap [j+k,b] = \emptyset$. But then
\[ |D \cap [m_i,m_i+n_i-1]| \leq \frac{k \epsilon n_i}{k+1} + 2k < \epsilon n_i + 2k, \]
a contradiction. \qed
\end{proof}

\begin{theorem}
\label{theorem:SubshiftNotClosed}
There exists a recursive two-way subshift $X$ such that $\tilde{X}$ is not closed in either of the global topologies.
\end{theorem}

\begin{proof}
One such subshift $X$ is the orbit closure of $\{{}^\infty a b^m c^n U(\frac{m}{n}) \;|\; m, n \in \N\}$, which is clearly recursive. The point $y = {}^\infty a U(\frac{\sqrt{2}}{2})$ is not in $\tilde{X}$. For otherwise, let $x \in X$ be such that $d(x, y) = 0$. Then $x$ is $a$-infinite to the left and must contain either a $0$ or a $1$. This means it is actually of the form ${}^\infty a b^m c^n U(\frac{m}{n})$. But then, the distance between the right tails is at least $|\frac{m}{n} - \frac{\sqrt{2}}{2}| > 0$. We also obtain arbitrarily good approximations for $y$ by taking points of this form with rational approximations of $\frac{\sqrt{2}}{2}$. \qed
\end{proof}

\section{Simplicial Complexes and Sofic Shifts}
\label{sec:SimpComp}

The following proposition is easy to prove with a construction similar to that in Definition~\ref{def:WPath}.

\begin{proposition}
\label{proposition:EmbeddingOfRN}
There is a G-embedding of $\R^\N$ to $\{0,1\}^\N$.
\end{proposition}

\begin{proof}
We begin with marking all the even coordinates as undefined, and injectively map $\R$ first onto $(0,1)$ using some standard function, and then into the odd coordinates using a $U$-like function. Then, we take the set of coordinates $i$ with $i \equiv 2 (\mod 4)$, and injectively map $\R$ into them as above. Inductively we obtain the desired embedding. \qed
\end{proof}

As a corollary, we obtain that all simplicial complexes can be embedded in $\Sigma^\Z$. In particular, every finitely generated group with a finite set of defining relations can be implemented as the fundamental group of some path-connected subset of $\Sigma^\Z$ \cite{Ma96}. However, the image in this map is not a very natural subset, and uses a very small slice of the space. Since $\R^\N$ and $\Sigma^\Z$ are not G-locally homeomorphic, this cannot be improved much. However, using the concept of homotopy equivalence we are able to prove a strong connection between simplicial complexes and B-sofic shifts, namely that they occupy the exact same homotopy equivalence classes. Since the fundamental group is an invariant of homotopy equivalence, the claims of the previous paragraph will still hold.

\begin{lemma}
\label{lemma:SFTInSofic}
If $X$ is a mixing sofic shift with positive entropy, there exists a mixing SFT $Y \subset X$ with positive entropy.
\end{lemma}

\begin{proof}
Let $w \sqsubset X$ be an unbordered synchronizing word (which always exists by \cite[Lemma 1]{BePe06}) and let $u, v$ be distinct words that do not contain $w$ such that $w u w \sqsubset X$, $w v w \sqsubset X$ and $|v| = |u| + 1$ (they exist since $X$ is mixing). Let $Y$ the subshift of $X$ whose points are concatenations of $wu$ and $wv$. Now $Y$ is a mixing SFT which clearly has positive entropy. \qed
\end{proof}

\begin{definition}
Let $V$ be a real vector space, and let $x \in V^n$ and $y \in \R^n$ with $y_i \neq 0$ for all $i$. The \emph{inverse $y$-weighted average of $x$} is defined as
\[ W(x,y) = \frac{\sum_{i=1}^n y_i^{-1} x_i}{\sum_{i=1}^n y_i^{-1}}. \]
\end{definition}

\begin{theorem}
\label{theorem:SimplicialComplexes}
Let $X$ be a mixing sofic shift with positive entropy and let $K$ be an abstract simplicial complex. Then there exists a sofic subshift $Y$ of $X$ such that $\tilde Y$ is B-homotopy equivalent to $K$.
\end{theorem}

\begin{proof}
By Lemma~\ref{lemma:SFTInSofic}, we may assume without loss of generality that $X$ is a mixing SFT. Let $K$ have $n$ vertices, and choose a realization for $K$ in some Euclidean space. Since block maps are continuous in the Besicovitch topology, we may recode $X$ into an edge shift. Let $w \sqsubset X$ be an unbordered word, and let $v, u_1, \ldots, u_n$ be distinct words that do not contain $w$ such that for some $k$, $w u_i w \sqsubset X$, $w v w \sqsubset X$, $|u_i| = k$ and $|v| = k + 1$. It is easy to see that such words always exist. Denote $U = \{u_1,\ldots,u_n\}$. We label the vertices of $K$ with $U$, and for all $S \in K$, we let $Y_S$ be the subshift of $X$ whose points are concatenations of $wv$ and the $wu$ for $u \in S$. Denote also $Y = \bigcup_{S \in K} Y_S$.

We define two continuous functions which form a homotopy equivalence between $K$ and $\tilde Y$. First, $f : K \to Y$ is defined using the average functions $A_{Y_S}$. We define it first on the $0$-skeleton of $K$ by $f(u) = (wu) \INF$ for all $u \in U$, and proceed by induction, assuming $f$ is already defined on the $(k - 1)$-skeleton. For each $k$-simplex $S$ with barycenter $z$, we choose $f(z) \in Y_S$ arbitrarily. Now, let $x \neq z$ be in the interior of $S$, and draw a line segment starting from $z$, going through $x$ and ending at the boundary of $S$. Let $y$ be the endpoint of the segment, and define
\[ f(x) = A_{Y_S}\left(\frac{d(z, x)}{d(z, y)}, f(z), f(y)\right). \]
Since the restriction of $f$ to the boundary of $S$ is continuous by the induction hypothesis, it is easy to see that the newly defined $f$ is continuous on the boundary. Clearly $f$ is also continuous at $z$, and it is continuous at other interior points of $S$ by the fact that $A_{Y_M}$ is continuous.

For the function $g : Y \to K$ we need the auxiliary concept of \emph{projection} from a shift $Y_S$ to one of its subshifts $Y_T$, where $T \subset S$. For this, choose an \emph{anchor point} $y \in Y_T$. The projection $\pi^S_T : Y_S \to Y_T$ is defined for a point $x \in Y_S$ by first letting $R = \{i \in \N \;|\; x_{[i - 2m, i + 2m]} \sqsubset Y_T\}$, and defining
\[ \pi^S_T(x)_i = \left\{\begin{array}{ll}
x_i, & \mbox{if } [i - m, i + m] \cap R \neq \emptyset \\
y_i, & \mbox{if } [i - 2m, i + 2m] \cap R = \emptyset
\end{array}\right.\]
and deterministically filling the rest (which can be done, since $m$ is a mixing distance for $Y_T$). It is easy to see that $\pi^S_T$ is continuous, and that $d_B(x,\pi^S_T(x))$ approaches $0$ if $d_B(x,Y_T)$ does.

Again, we proceed by induction on the dimension of the simplex. Define $g(y) = u$ for all $y \in Y_u$ and $u \in U$, and assume that $g$ has been defined continuously on the union of all $Y_S$ with $S$ in the $(k-1)$-skeleton of $K$. Let $S$ be an arbitrary $k$-dimensional simplex, and let $S_i$ for $i \in [1,k]$ be its $(k - 1)$-skeleton. Denote $\pi_i = \pi_{S_i}$ for all $i$. Let $x \in Y_S - \bigcup_{i=1}^k Y_{S_i}$, and denote $g_i = g \circ \pi_i$ and $d_i(x) = d_B(x,\pi_i(x))$ for all $i$. Note that $d_i(x) > 0$ for all $i$ by Theorem~\ref{theorem:SoficClosed}. We then extend $g$ to $Y_S$ by defining
\[ g(x) = W \! \left( (g_i(x))_{i=1}^k, (d_i(x))_{i=1}^c \right). \]
Now $g$ is continuous on the interior of $Y_S$, since the $\pi_i$, $d_B$ and $W$ are, and the restriction of $g$ to the $Y_{S_i}$ is by the induction hypothesis. On the $Y_{S_i}$, the extended $g$ is continuous, since the inverted average $W$ will approach the value $g_i(x)$ as $x$ approaches the subshift $Y_{S_i}$.

Now $g \circ f : K \to K$ and the continuously extended version of $f \circ g : Y \to Y$ are homotopic to the respective identity maps by Lemma~\ref{lemma:LiesImplyHomotopy} and Lemma~\ref{lemma:BHomotopy}. \qed
\end{proof}

\begin{corollary}
Every finitely generated group with a finite set of defining relations can be implemented as the fundamental group of a sofic subshift of any mixing SFT in the Besicovitch topology.
\end{corollary}

We then prove a converse for Theorem~\ref{theorem:SimplicialComplexes}.

\begin{theorem}
Let $X$ be a sofic shift. Then there exists a simplicial complex $K$ which is B-homotopy equivalent to $\tilde X$.
\end{theorem}

\begin{proof}
Let $\mathcal{T}$ be the set of transitive components of $X$. For all $R \subset \mathcal{T}$, we define the \emph{intersection set} $Y_R = \bigcap R$ and the \emph{B-intersection set} $Z_R = \bigcap_{Y \in R} \tilde Y$. Note that each $Y_R$ is either a sofic shift or the empty set, and $Z_R$ is empty if and only if $Y_R$ is (although, in general, these sets look very different). We let $I$ be the set of pairs $(Y_R,Z_R)$ such that $Y_R \neq \emptyset$ ordered by inclusion in the first coordinate.

For each $P = (Y,Z) \in I$, we choose a point $x_P \in Z$. Let $V$ be the set of minimal elements of $I$. We define an abstract simplicial complex $K$ by
\[ K = \{ J \subset V \;|\; \sup J \mbox{ exists in } I \}.\]
Label each simplex $J$ of $K$ with $\sup J \in I$. We now prove that $X$ and $K$ are homotopy equivalent by explicitly constructing two continuous functions $f : K \to X$ and $g : X \to K$ that form the equivalence.

As in Theorem~\ref{theorem:SimplicialComplexes}, $f$ is defined by induction on $k$ on the $k$-skeletons of $K$. For all vertices $P \in V$, $f$ simply maps the corresponding point in $K$ to $x_P$. On simplexes $S \subset K$ of higher dimension, $f$ is defined as in Theorem~\ref{theorem:SimplicialComplexes}, but mapping the barycenter to $x_P$ for some $P < S$. The averaging can be done, since any two points $x_P, x_Q$ with $P,Q < S$ must be in a common transitive component, and $f$ is continuous as in Theorem~\ref{theorem:SimplicialComplexes}. If $S$ has label $(Y,Z) \in I$, then $f(S) \subset Z$, since B-intersection sets are closed under the average functions $A$.

To each simplex $S$ with label $P = (Y,Z)$, we define a function $g_P : X \to S$ with the property that $g_P(Z') \subset T$ for all simplexes $T \subset S$ with label $(Y',Z')$. If $P \in V$, then $g_P$ is simply the constant map from $X$ to the vertex $P$. Suppose then that $P \notin V$, and let $\mathcal{Q} = \{ Q \in I \;|\; Q < P \}$. If $x \in Z'$ for some $Q = (Y',Z') \in \mathcal{Q}$, define $g_P(x) = g_Q(x)$. Otherwise, denote $h_P(x) = (g_Q(x))_{Q \in \mathcal{Q}}$ and $d_P(x) = (d(x, Z))_{(Y,Z) \in \mathcal{Q}}$, and define
\[ g_P(x) = W(h_P(x), d_P(x)). \]
Since every $x \in X$ lies in some B-intersection set $Z$, we may define $g(x) = g_{(Y,Z)}(x)$ for the pair $(Y,Z) \in I$, and this value does not depend on the choice of $Z$. It is easy to see that $g$ is continuous, and that $g_{(Y,Z)}(Z) \subset S$ holds for all simplexes $S$ with label $(Y,Z)$.

We claim that the functions $f$ and $g$ form a homotopy equivalence between $\tilde X$ and $K$. Namely, if a simplex $S \in K$ corresponds to a transitive component $T \in \mathcal{T}$, then $f(g(T)) \subset \tilde T$ and $g(f(S)) \subset S$. Thus $f \circ g$ and $g \circ f$ satisfy the requirements of Lemma~\ref{lemma:BHomotopy} and Lemma~\ref{lemma:LiesImplyHomotopy}, and the claim follows. \qed
\end{proof}

We believe that the above results hold also in the Weyl topology, but with slightly more complicated proofs. Namely, the averaging functions $A_X$ have to be defined using $U$ instead of $U'$, and the dissection of $X$ to the $Y_R$ and $Z_R$ has to be more intricate.

It is easy to see that not all subshifts are homotopy equivalent to a simplicial complex:

\begin{example}
Let $X$ be the orbit closure of $\{{}^\infty(10^n)^\infty \;|\; n \in \N\}$. Clearly, $\tilde X$ has a countably infinite number of G-path-connected components, so it cannot be G-homotopy equivalent to any simplicial complex.
\end{example}

Unlike the previous example and the easily proven fact that zero-entropy sofic shifts always occupy finitely many $\sim$-classes, zero-entropy subshifts can have rather complicated G-structure:

\begin{example}
\label{example:WeirdZeroEntropy}
Let $X$ be the C-orbit closure of $U'([0,1])$, a zero-entropy subshift, and denote by $Y$ the projection of $X$ to the quotient topology defined by $\sim_B$. Every point of $Y$ is a fixed point of $\sigma$, so $U'$ induces a homeomorphism between $[0,1]$ and the projection of $U'([0,1])$. By adding the projections of the points $\INF 0 \INF, \INF 1 \INF, \INF 01 \INF$ and $\INF 10 \INF$ as necessary, every subset of $[0,1]$ is homeomorphic via $U'$ to the projection of a subshift of $X$ modulo at most $4$ isolated points.
\end{example}

Using the above, we present an unrelated solution to an open problem of \cite{CaFoMaMa97}.

\begin{theorem}
There are nontrivial shift-invariant B-Borel measures on $\Sigma^\Z$.
\end{theorem}

\begin{proof}
The map in Example~\ref{example:WeirdZeroEntropy} can be used to transport measures on $[0, 1]$ to those on the zero-entropy subshift. \qed
\end{proof}

\section{Projections into a Subset of the Full Shift}
\label{sec:Projections}

The proof of Theorem~\ref{theorem:SimplicialComplexes} uses the interesting idea of projecting points of a shift space onto a smaller subshift, but the projection used was rather arbitrary. In this section, we show that projections from the full shift cannot be made canonical in a natural sense for any subshift.

\begin{definition}
A subset $X$ of $\Sigma^\Z$ is said to have the \emph{unique approximation property} if for every $y \in \Sigma^\Z$ the set $\{ x \in X \;|\; d_G(y,x) = d_G(y,X) \}$ consists of a single $\sim$-equivalence class, denoted $\pi_X(y)$.
\end{definition}

Obviously a set with the unique approximation property needs to be G-closed.

\begin{lemma}
\label{lemma:ALemma}
Let a subshift $X$ have the unique approximation property. Then, for each $x \in \Sigma^\Z$ with period $p$, the class $\pi_X(x)$ contains a $p$-periodic point of $X$.
\end{lemma}

\begin{proof}
Let $y \in \pi_X(x)$, and note that necessarily $\sigma^p(y) \sim y$, since $\sigma$ is an isometry for $d_G$. Then there exists a word $w \in \Sigma^p$ such that $w^n \sqsubset y$ for arbitrarily large $n$, and $d_G(x,\INF w \INF) = d_G(x,X)$. But now $\INF w \INF \in \pi_X(x) \cap X$, proving the claim. \qed
\end{proof}

In particular, the above lemma implies that if there exist two $\sim$-distinct points in $X$ whose distance to $x$ is $p^{-1}$, then $x \in X$.

\begin{theorem}
Let $X$ be a nonempty subshift with the unique approximation property, and assume that $\tilde X$ consists of at least two $\sim$-classes. Then $X = \Sigma^\Z$.
\end{theorem}

\begin{proof}
We will show, by induction on $i \in \N$, that $X$ contains all points with period $2^i$, from which the claim follows. Begin with a unary point $y \in \Sigma^\Z$, and note that Lemma~\ref{lemma:ALemma} implies the existence of a unary point $x \in \pi_X(y)$. If $y \notin X$, then by definition $d(y,X) = d(y,x) = 1$. But now, any other $\sim$-class of $\tilde X$ could be chosen as $\pi_X(y)$, a contradiction.

Choose a symbol $0 \in \Sigma$. For the induction step, we will prove by a new induction on the number of non-$0$ symbols in a word $w \in S^{2^i}$ that $\INF w \INF \in X$. First, let $w$ contain a single non-$0$ symbol. We may assume that $w = 0^{2^{i-1}}v$ for some $v \in \Sigma^{2^{i-1}}$. Now, $\INF v \INF$ and $\INF 0 \INF$ are $2^i$-periodic points of $X$ whose distance to $\INF w \INF$ is $2^{-i}$, so Lemma~\ref{lemma:ALemma} implies that $\INF w \INF \in X$. For an arbitrary $w$, we find two points of $X$ whose distance to $\INF w \INF$ is $2^{-i}$ by switching two different non-$0$ symbols in $w$ to $0$. Both of these points are in $X$ by the induction hypothesis, so Lemma~\ref{lemma:ALemma} again implies that $\INF w \INF \in X$. \qed
\end{proof}

\begin{example}
The set $\{{}^\infty\{00, 01\}.\{00, 01\}^\infty\}$ has the unique approximation property.
\end{example}

A straightforward argument also shows that if a point of the full shift has two best approximations in a mixing SFT, then it has an uncountable number of them.

\section{Geometry of Cellular Automata}
\label{sec:CA}

We start by proving that so-called contracting maps have a very simple structure in a full shift.

\begin{definition}
We say a CA $f$ defined on a subshift $X$ is \emph{contracting} if for all points $x, y \in X$ we have $d(x, y) \geq d(f(x), f(y))$, \emph{expanding} if $d(x, y) \leq d(f(x), f(y))$, and an \emph{isometry} if $d(x, y) = d(f(c), f(d))$.
\end{definition}

\begin{proposition}
\label{proposition:ContractingHasRadiusOne}
In the full shift, a CA whose global map is G-contracting must have a size $1$ neighborhood.
\end{proposition}

\begin{proof}
Let $f$ be a CA with minimal connected neighborhood of size $r > 1$. Since the neighborhood is minimal, there exist $u$ and $v$ with $|u| = |v| = r - 1$ such that $f(ua) \neq f(ub)$ for some $a, b \in \Sigma$, and $f(cv) \neq f(dv)$ for some $c, d \in \Sigma$. Let $M$ be a $\Sigma \times \Sigma$-matrix over $2^\Sigma$ defined by \[ s \in M_{i,j} \iff (f(us), f(sv)) = (i, j). \]

Clearly, every $s \in \Sigma$ occurs in exactly one coordinate $(i, j)$ of $M$. Since $f(ua) \not= f(ub)$, $a$ and $b$ occur in different rows in the matrix $M$. Similarly, $c$ and  $d$ occur in different columns. Therefore, one of the pairs $\{a, c\}$, $\{a, d\}$, $\{b, c\}$ or $\{b, d\}$, say $\{a, c\}$, has the property that $a$ and $c$ occur in different rows and different columns. Then $f(u a v)$ and $f(u c v)$ differ in at least two coordinates.

Now consider the configurations $x = \INF (u a v) \INF$ and $y = \INF (u c v) \INF$. We have $d_G(x, y) = \frac{1}{2r - 1}$, but by the previous discussion, $d_G(f(x), f(y)) \geq \frac{2}{2r - 1}$, and $f$ cannot be contracting. \qed
\end{proof}

\begin{corollary}
\label{corollary:TrivialIsometries}
In the full shift, a CA whose global map is a G-isometry is a composition of a shift and a symbol permutation.
\end{corollary}

\begin{corollary}
\label{corollary:ExpandingIsIsometry}
In the full shift, a CA whose global map is G-expanding is a G-isometry.
\end{corollary}

\begin{proof}
An expanding map must be injective, so such a CA $f$ must have an inverse $f^{-1}$. Since $f$ is G-expanding, $f^{-1}$ is a G-contracting injective map, and thus has radius $1$, and must in fact be a G-isometry. Then $f$ is also a G-isometry. \qed
\end{proof}

We can extend Corollary~\ref{corollary:TrivialIsometries} as follows (using a completely different proof):

\begin{theorem}
\label{theorem:IsometriesOnNaturalSpaces}
Let $X \subset \Sigma^\Z$ be a subshift such that for some $0 \in \Sigma$ we have ${}^\infty 0^\infty \in X$, and for all $w \sqsubset X$ and all symbols $s \sqsubset w$ there is a $p$-periodic point $x \in X$ such that $w \sqsubset x$, and ${}^\infty(s0^{p-1})^\infty \in X$. Then, all G-isometric CA on $X$ have neighborhood size $1$.
\end{theorem}

\begin{proof}
Suppose that a CA $f : X \to X$ is a G-isometry, and denote $f({}^\infty 0^\infty) = {}^\infty a^\infty$. Let $p \in \N$. Every point of the form ${}^\infty(s0^{p-1})^\infty$ with $s \in S$ is mapped to some translate of ${}^\infty(ta^{p-1})^\infty$ with $t \in S$. By comparing the images of distinct points, we see that $f$ acts on them as some shift $\sigma^{n_p}$ composed with some symbol permutation $g_p$. Also, $n_p$ and $g_p$ are constants (denote them by $n$ and $g$) when $p$ is greater than the radius $r$ of $f$.

Let then $w \sqsubset X$, and write it in the form $usv$ for some $u,v \sqsubset X$ and $s \in S$. By the assumption on $X$, there exist $p$-periodic points $x = {}^\infty (svtu)^\infty$ and $y = {}^\infty(s0^{p-1})^\infty$ in $X$ for some $t \sqsubset X$ and $p > r$. By comparing $f(x)$ and ${}^\infty a^\infty$, we see that the percentage of $0$'s in $x$ equals the percentage of $a$'s in $f(x)$. Then, comparing $f(x)$ and $f(y) = \sigma^n(g(y))$, we see that the symbol $g(s)$ must appear in $f(x)$ exactly $n$ positions from the place of $s$ in $x$. Since $s$ was an arbitrary symbol appearing in $w$, we see that $f = \sigma^n \circ g$. \qed
\end{proof}

In particular, isometric cellular automata have neighborhood size $1$ on the golden mean shift and the even shift. In fact, so do all contracting CA on the golden mean shift, as the proof of Proposition~\ref{proposition:ContractingHasRadiusOne} works also in this case.

\begin{example}
The previous theorem does not hold as such for contracting CA. Namely, let $X \subset \{0,1\}^\Z$ be the SFT where $111$ is forbidden (which clearly satisfies the conditions of the theorem), and let $f$ be the CA that transforms every occurrence of $01$ to $00$. Then $f$ is contracting, since the only case in which $x_i = y_i$ but $f(x)_i \neq f(y)_i$ holds is $\begin{smallmatrix} x & = & \ldots & 1 & 1 & \ldots \\ y & = & \ldots & 0 & 1 & \ldots \end{smallmatrix}$. But in this case the upper-left $1$ is preceded by a $0$, and $f(x)_{i-1} = f(y)_{i-1}$, so that the offending difference is canceled.
\end{example}

\begin{example}
A CA on the disjoint union of two full shifts acting as the identity on one and as $\sigma$ on the other is isometric, but has neighborhood size $2$.
\end{example}

\begin{example}
Consider the product of $\{0,1\}^\Z$ with itself, the symbol permutation $f : (x,y) \mapsto (y,x)$ and the partial shift $g = \ID \times \sigma$. Then $f$ is clearly G-isometric, but $g \circ f \circ g^{-1}$ is not by Theorem~\ref{theorem:IsometriesOnNaturalSpaces}. Thus G-isometry is not a conjugacy invariant.
\end{example}

\begin{question}
What is the closure of the class of G-isometric cellular automata under conjugacy?
\end{question}

\begin{definition}
We say a B-continuous B-shift-commuting map $f$ on a subshift $X$ is \emph{CA-like} if for each $x \in X$, there exists a cellular automaton $g_x$ such that $f(x) \sim_B g_x(x)$.
\end{definition}

\begin{theorem}
A CA-like function $f$ on $\Sigma^\Z$ is a CA.
\end{theorem}

\begin{proof}
We show that $f$ behaves as the same cellular automaton on all $x \in \Sigma^\Z$. For this, let $z$ be a generic point for the uniform Bernoulli measure (see \cite{Wa82}), which means that for all $w \in \Sigma^*$ we have
\[ \lim_{n \longrightarrow \infty} \frac{1}{2n+1}|\{i \in [-n,n] \;|\; z_{[i,i+|w|-1]} = w\}| = |\Sigma|^{-|w|}. \]
Let $p : [0, 1] \to S^\Z$ be a path from $x$ to $z$ defined by $p(r) = A_{\Sigma^\Z}(r,x,z)$ (where the definition of $A$ is extended to $\Sigma^\Z$ in the obvious way).

For all $r \in (0, 1]$ the cellular automaton $g_{p(r)}$ is uniquely defined by the genericity of $z$. We define a topology for cellular automata by the pseudometrics
\[ d_w(a, b) = \left\{\begin{array}{rl} 0 & \mbox{if } a({}^\infty0w0^\infty)_0 = b({}^\infty0w0^\infty)_0 \\ 1 & \mbox{otherwise} \end{array}\right. \]
for bidirectional finite words $w$, where $0 \in \Sigma$ is fixed. Denote the space of all CA on $\Sigma^\Z$ with this topology by $\CA$.

Consider a point $r \in (0,1]$ and a word $w \in \Sigma^*$, and note that $w$ occurs in $p(r)$ with asymptotic density at least $\alpha = r|S|^{-|w|}$. Let $\epsilon$ be such that $f$ changes by at most $\alpha/2$ in the $\epsilon$-neighborhood of $p(r)$, and let $\delta$ be such that the image of $p$ changes by at most $\min(\epsilon, \alpha/2)$ in the neighborhood $(r - \delta, r + \delta)$. Then for all automata $h \in \{g_{p(t)} | r - \delta \leq t \leq r + \delta\}$, we have that $d_w(h, g_{p(r)}) = 0$. This is because $w$ occurs with density at least $\alpha/2$ in $p(t)$, and thus changing the image of $w$ would result in $d(f(p(r)), f(p(t))) > \alpha/2$, a contradiction. Thus, the function $g \circ p : (0, 1] \to \CA$ is continuous.

If $g \circ p$ is not a constant function, we have obtained a nontrivial path in a countable $\mathrm{T}_1$ space $\CA$. This, however, is impossible, since then the preimages of singletons give a partition of an interval $[s, t]$ into a countable number of closed sets, which is a contradiction by a straightforward compactness argument. Thus $g \circ p$ is a constant map whose image is the automaton $g_z$, and then it is easy to see that $f(x) \sim_B g_z(x)$. \qed
\end{proof}

\begin{question}
Is this also true in the Weyl topology?
\end{question}

\section{Nonexistence of Transitive Cellular Automata}
\label{sec:DynamicalProperties}

Let us show how the nontransitivity of CA follows from basic measure theory and a pigeonhole argument. For this, we need a lemma.

\begin{lemma}[\cite{St01}]
\label{lemma:Binomials}
For all $n,m,p \in \N$ we have
\[ \binom{mn}{pn} < \frac{1}{\sqrt{2\pi}}n^{-\frac{1}{2}}\frac{m^{mn+\frac{1}{2}}}{(m - p)^{(m - p)n + \frac{1}{2}} \cdot p^{pn + \frac{1}{2}}} \]
\end{lemma}

To make this lemma more easily usable in the proof of Theorem~\ref{theorem:NoTransitiveCA}, we set $p = 1$ and compare with an exponential function.

\begin{lemma}
\label{lemma:BinomialUpperBound}
For all $1 < k \in \R$ and for all $0 < a \in \R$, we have that for some $m$ and all large enough $n$ divisible by $m$,
\[ \binom{n}{n/m} \leq k^{na}   \]
\end{lemma}

\begin{proof}
When the conditions are satisfied, Lemma~\ref{lemma:Binomials} gives
\begin{align*}
\binom{n}{n/m} &< \frac{1}{\sqrt{2\pi}}(n/m)^{-\frac{1}{2}}\frac{m^{n+\frac{1}{2}}}{(m - 1)^{\frac{m - 1}{m}n + \frac{1}{2}}} \\
&\leq \frac{1}{\sqrt{2\pi}}(n/m)^{-\frac{1}{2}}m^{(n/m)}\left(\frac{m}{m-1}\right)^{\frac{m - 1}{m}n+\frac{1}{2}} \\
&\leq k^{n\frac{a}{3}} \left(m^{1/m}\right)^n\left(\frac{m}{m-1}\right)^{n+\frac{1}{2}} \leq k^{na}
\end{align*} \qed
\end{proof}

For a word $w \in \B_n(X)$, we denote $[w] = \{ x \in X \;|\; x_{[-\lfloor n/2 \rfloor,\lceil n/2 \rceil]} = w \}$. If $W \subset \B_n(X)$, we denote $[W] = \bigcup_{w \in W} [w]$. For all $X \subset \Sigma^\Z$ we define $\B_n(X) = \{ w \in \Sigma^n \;|\; w \sqsubset X \}$ and $\B^0_n(X) = \{ w \in \Sigma^n \;|\; [w] \cap X \neq \emptyset \}$. We also denote
\[ d_n(x, y) = \sup_{m \geq n} \left\{\frac{H\left(x_{[-m, m]}, y_{[-m, m]}\right)}{2m + 1} \right\} \]
and note that $d(x, y) = \lim d_n(x, y)$. For all $\epsilon > 0$ and $x \in X$, denote $B^n_\epsilon(x) = \{ y \in X \;|\; d_n(x,y) < \epsilon \}$, and note that $B^n_\epsilon(x) \subset B^m_\epsilon(x) \subset B_\epsilon(x)$ whenever $n \leq m$.

\begin{theorem}
\label{theorem:NoTransitiveCA}
No cellular automaton is B-transitive on a positive entropy sofic shift $X \subset \Sigma^\Z$.
\end{theorem}

\begin{proof}
First, we note that such a CA would have to be surjective by Theorem~\ref{theorem:SoficClosed}.

Assume on the contrary that $f : X \to X$ is a surjective transitive CA, and let $\mu$ be a probability measure of $X$ given by a Markov measure on its Shannon cover (see \cite{LiMa95}). Then there exist $\gamma \in (0,1)$ and $t \geq 1$ such that for all $n \in \N$ and $w \in \B_{tn}(X)$ we have $\mu([w]) \leq \gamma^n$.
Let $\delta < \frac{1}{2}$ be such that $|\Sigma|^{\delta (t+1)} \gamma < 1$.

Now let $\epsilon$ and $\epsilon'$ be such that $2\epsilon + 2\epsilon' < \delta$, and the requirements of Lemma~\ref{lemma:BinomialUpperBound} are satisfied for $k = |\Sigma|$, $a = \epsilon'$ and $m \leq \epsilon^{-1}$. Let $x = \INF w \INF \in X$ for some $w \in \B_p(X)$, and denote $B^n = B^n_\epsilon(x)$.

For all $n,l \in \N$, define $X(n, l) = B^n_\epsilon(f^l(B^n))$. 
From the definitions of transitivity and $d_n$, it follows that $\bigcup_{n, l} X(n, l) = X$.
Since $X(n, l)$ is clearly measurable for all $n, l \in \N$, some $Y = X(n_0, l_0)$ must have positive $\mu$-measure. Define $g = f^{l_0}$ and $B = B^n \cap g^{-1}(B^n_\epsilon(Y))$, and let $r$ be the radius of $g$.

Since $\mu(Y) > 0$, we must have $|\B^0_n(Y)| \geq |\Sigma|^{\delta n}$ for all large enough $n$. If not, we have for arbitrarily large $n = tq + t'$ where $t' \in [0, t - 1]$ that
\[ \mu(Y) \leq \mu([\B^0_{tq + t'}(Y)]) \leq \mu([\B^0_{tq}(Y)]) \leq |\Sigma|^{\delta (tq + t')} \gamma^q < (|\Sigma|^{\delta (t+1)} \gamma)^q, \]
which is a contradiction, since the rightmost term is then arbitrarily small.

We define $D_n = \{ v \in \B_n(X) \;|\; \exists u \in \B_n(x) : H(v,u) \leq n\epsilon \}$. Then for all $n \geq n_0$ we have $\B^0_n(B) \subset D_n$, which implies that $|\B^0_n(g(B))| \leq |\Sigma|^{2r}|D_n|$. Since we also have $|D_n| \leq p\binom{n}{\lfloor n \epsilon \rfloor}|\Sigma|^{n\epsilon}$ and $Y \subset B^n_\epsilon(g(B))$, this implies that
\[ |\B_n^0(Y)| \leq \binom{n}{\lfloor n \epsilon \rfloor}|\Sigma|^{n\epsilon}|\B^0_n(g(B))| \leq p \binom{n}{\lfloor n \epsilon \rfloor}^2|\Sigma|^{2r+2n\epsilon}, \]
which is at most $p|\Sigma|^{2r+2n\epsilon+2n\epsilon'}$ by Lemma~\ref{lemma:BinomialUpperBound}. But if $n$ is large enough, we have $|\B^0_n(Y)| \geq |\Sigma|^{\delta n}$, and thus
\[ 1 = |\B_n^0(Y)|/|\B_n^0(Y)| \leq p|S|^{2r + n(2\epsilon + 2\epsilon' - \delta)}, \]
which converges to $0$ as $n$ grows. We have reached a contradiction. \qed
\end{proof}

\bibliographystyle{plain}
\bibliography{//utuhome.utu.fi/iatorm/Stuff/AlgTop/bib}{}

\end{document}